\documentclass[1pt]{article}
\usepackage{amsmath, amssymb, amsthm, amsfonts, cases}
\usepackage{mathrsfs}
\usepackage{url}
\usepackage{authblk}
\usepackage[usenames]{color}
\usepackage{geometry}
\allowdisplaybreaks

\theoremstyle{plain}
\newtheorem{thm}{Theorem}[section]

\newtheorem{pro}[thm]{Problem}
\newtheorem{lem}[thm]{Lemma}

\theoremstyle{definition}
\newtheorem{defn}{Definition}[section]
\newtheorem{ass}{Assumption}[section]

\newcommand{\eps}{\varepsilon}

\newcommand{\la}{\langle}
\newcommand{\ra}{\rangle}

\makeatletter\@addtoreset{equation}{section} \makeatother
\begin{document}

\title{  Stochastic Evolution Equations of Jump Type with Random Coefficients: Existence, Uniqueness and Optimal Control
 \thanks{This work was supported by the Natural Science Foundation of Zhejiang Province
for Distinguished Young Scholar  (No.LR15A010001),  and the National Natural
Science Foundation of China (No.11471079, 11301177) }}

\date{}

   \author{Maoning Tang
 \hspace{1cm}  Qingxin Meng\thanks{Corresponding author.   E-mail: mqx@zjhu.edu.cn}
\hspace{1cm}
\\\small{Department of Mathematics, Huzhou University, Zhejiang 313000, China}}

\maketitle

\begin{abstract}

We study a class of stochastic evolution 
equations of jump type with random coefficients and its optimal control problem. There are
three major ingredients. The first is to 
prove 
the existence and uniqueness of the solutions by continuous dependence theorem of solutions combining with  the  parameter extension method. The second is
to establish the stochastic maximum principle
and verification theorem for our optimal control problem by the classic 
convex variation method and dual technique.
The third is to represent an example of   a Cauchy problem for a controlled stochastic partial differential equation with jumps which our theoretical results can solve.

\end{abstract}


\maketitle

\section{ Introduction }

In this paper,  we study the  following
stochastic evolution equation with
jump  \begin{eqnarray} \label{eq:1.1}
  \left\{
  \begin{aligned}
   d X (t) = & \ [ A (t) X (t) + b ( t, X (t)) ] d t
+ [B(t)X(t)+g( t, X (t)) ]d W(t)
 +\int_E \sigma (t, X(t),e)\tilde \mu(de,dt),  \\
X (0) = & \  x , \quad t \in [ 0, T ]
  \end{aligned}
  \right.
\end{eqnarray}
in the framework of a Gelfand triple $V \subset H= H^*\subset V^{*},$  where $ H$, $V$ are Hilbert spaces. Here $W$  is  a  one-dimensional  Brownian motion  and  $\tilde \mu$
is a Poisson random martingale measure
on  a filtrated probability  $(\Omega, \mathscr{F},\{
{\mathscr F}_t\}_{0\leq t\leq T}, P),$
$A:[0,T]\times \Omega \longrightarrow {\mathscr L} (V, V^*)$, $B
  : [0,T]\times \Omega \longrightarrow {\mathscr L} (V, H ),$ $b:[0,T]\times  \Omega\times H  \longrightarrow H$, $g:[0,T]\times\Omega\times H  \longrightarrow H$ and $\sigma:[0,T]\times  \Omega \times
   E\times H  \longrightarrow H$  are given random mappings. Here we denote by  $\mathscr{L}(V,V^*)$ the space of bounded
linear transformations of V into $V^*$, by ${\mathscr L} (V, H)$   the space of bounded
linear transformations of  $H$ into $V$.
    An adapted solution of
   \eqref{eq:1.1} is  a $V$-valued, $\{{\mathscr F}_t\}_{0\leq t\leq T}$-adapted process $X(\cdot)$
    under some appropriate sense.

Stochastic partial differential equations and stochastic evolution equations
driven  only by Wiener processes  have been
 investigated in depth  and a great deal of advances have been made by many authors, see, for example the monographs [5],[6],[23]  for their general
 theory. Most recently, thanks to comprehensive practical applications,  many attentions have been paid to stochastic partial differential equations driven by jump processes, (cf., for example,[1],[2],[4],[10],[21],[22],[25],[26],[28],
[29]  and the references therein).    It is worth mentioning that R\"{o}cker and Zhang [23]  established the existence and uniqueness theory for solutions of stochastic evolution equations of type (1.1) by a successive approximations, in which case  the operator does not exist and  all coefficients involved  are deterministic mappings.

The purpose of this paper is to show
 the existence and uniqueness of solutions to  the stochastic evolution equation \eqref{eq:1.1}
 and establish the corresponding
 maximum principle and verification theorem
 for the optimal control problem where the state is driven by a  controlled stochastic evolution equation \eqref{eq:1.1}.
  The main feature of this paper is that all coefficients of \eqref{eq:1.1} are allowed to be random and time varying. Moreover, different from Picard iteration approach  adopted by R\"{o}cker and Zhang [25] to prove the existence and uniqueness results, our approach is that we first establish  continuous dependence theorem of solutions and then combine  the  parameter extension method to construct a contractive mapping.  For  optimal control problems of stochastic evolution equation or stochastic partial differential equation, many works are concerned with systems driven only by Wiener  process  and the corresponding stochastic maximum
principles are establish, see e.g.[3],[7],[9],[11],[12],[14],[15],[16],[17],
[18],[20],[30].
 In contrast, there have not been very much
  results on  the optimal control for stochastic partial
differential equations driven by jump processes. In 2005, Oksendal etl [23]  studied
the optimal control problem for a stochastic
reaction diffusion equation driven by Poisson random measure  in which
case  the operator $B$ is absence and all the coefficients are deterministic.
 Thus  our system  cover
 the model studied by {\O}ksendal etl [23] and  is more general  in such a way that all coefficients are random and time-varying.

The rest of this
paper is structured as follows. In Section 2,
 we provide the basic notations and recall It\^{o} formula for jump diffusion in Hilbert space used  frequently in this paper.
 Section 3  establishes the existence and
 uniqueness of solutions to the
 stochastic evolution equation \eqref{eq:1.1}.
 Section 4 formulates the optimal control problem specifying the hypotheses. In section 5, adjoint equation is introduced which turns out to be a backward stochastic
evolution equation with jumps. In Section 6, we establish the  stochastic  maximum principle by the classic convex variation method. In Sections 7, the verification theorem for optimal controls  is obtained by dual technique. In section 8, we present one example of application of our results.

\section{Notations and It\^{o} Formula for Jump Diffusion in Hilbert Space  }
In this section, we first introduce the notations which will be  used in our paper.
Let $(\Omega, \mathscr{F}, P)$ be a complete probability space
equipped with a  one-dimensional standard Brownian motion $\{W(t),
0\leq t\leq T\}$  and  a stationary Poisson point process
$\{\eta_t\}_{t\geq 0}$ defined  on a fixed nonempty Borel measurable
subset ${E}$ of $\mathbb R^1$.
 Denote by $\mathbb E[\cdot]$ the expectation
under the probability $\mathbb P.$
 We denote by $\mu(de,dt)$
 the counting measure induced by $\{\eta_t\}_{t\geq 0}$ and
  by  $\nu(de)$ the corresponding
 characteristic measure.  Then the compensatory
 random martingale measure is denoted by
  $\tilde{\mu}(de, dt):={\mu}(de,
dt)-\nu(de)dt$ which is assumed to be independent of the Brownian
motion.
  Furthermore, we assume that
$\nu({E})<\infty$. Let $\{{\mathscr F}_t\}_{0\leq t\leq T}$ be the
P-augmentation of the natural filtration generated by
$\{{W_t}\}_{t\geq 0}$ and $\{\eta_t\}_{t\geq 0}$.
 By  $\mathscr{P}$  we denote the
predictable $\sigma$ field on $\Omega\times [0, T]$ and
by $\mathscr B(\Lambda)$
  the Borel $\sigma$-algebra of any topological space $\Lambda.$
  Let $X$ be  a  separable Hilbert space with norm $\|\cdot\|_X$.
  Denote by  $M^{\nu,2}( E; X)$ the set of all $X$-valued measurable
  functions $r=\{r(e), e\in E\}$ defined on the measure
  space $(E, \mathscr B(E); v)$
  such that
$\|r\|_{M^{\nu,2}( E; X)}\triangleq
\sqrt{{\int_E\|r(e)\|_X^2v(de)}}<~\infty,$ by
${M}_\mathscr{F}^{\nu,2}{([0,T]\times  E; X)}$ the  set of all
$\mathscr{P}\times {\mathscr B}(E)$-measurable $X$-valued processes
$r=\{r(t,\omega,e),\
(t,\omega,e)\in[0,T]\times\Omega\times E\}$ such that
$\|r\|_{{M}_\mathscr{F}^{\nu,2}{([0,T]\times  E; X)}}\triangleq
\sqrt{{\mathbb E\bigg[\int_0^T\int_E\displaystyle\|r(t,e)\|_X^2
\nu(de)dt\bigg]}}<~\infty,$
 by
$M_{\mathscr{F}}^2(0,T;X)$ the set of all ${\mathscr{F}}_t$-adapted
$X$-valued  processes $f=\{f(t,\omega),\
(t,\omega)\in[0,T]\times\Omega\}$ such that
$\|f\|_{M_{\mathscr{F}}^2(0,T;X)}
\triangleq\sqrt{\mathbb E\bigg[\displaystyle\int_0^T\|f(t)\|_X^2dt\bigg]}<\infty,$
by  $S_{\mathscr{F}}^2(0,T;X)$  the set of all
${\mathscr{F}}_t$-adapted  $X$-valued c\`{a}dl\`{a}g processes
$f=\{f(t,\omega),\ (t,\omega)\in[0,T]\times\Omega\}$ such that
$\|f\|_{S_{\mathscr{F}}^2(0,T;X)}\triangleq\sqrt{
\mathbb E\bigg[\displaystyle\sup_{0
\leq t \leq T}\|f(t)\|_X^2}\bigg]<+\infty,$  by
$L^2(\Omega,{\mathscr{F}},P;X)$ the set of all $X$-valued random
variables $\xi$ on $(\Omega,{\mathscr{F}},P)$ such that
$\|\xi\|_{L^2(\Omega,{\mathscr{F}},P;X)}\triangleq
\sqrt{\mathbb E[\|\xi\|_X^2]}<\infty.$
Throughout this paper, we let  $C$ and $K$  be two generic positive constants, which may be different from line to line.

Let $V$ and $H$ be two separable (real) Hilbert spaces such that $V$
is densely embedded in $H$. We identify $H$ with its dual space by
the Riesz mapping. Then  we can take $H$ as a pivot space and  get a
Gelfand triple $V \subset H= H^*\subset V^{*},$ where  $H^*$ and
$V^{*}$ denote the dual spaces of $H$ and $V$, respectively. Denote
by $\|\cdot\|_{V},\|\cdot\|_{H}$ and $\|\cdot\|_{V^*}$ the norms of
$V,H$ and $V^*$, respectively, by $(\cdot,\cdot)_H$ the inner
product in $H$, by $\la\cdot,\cdot\ra$ the duality product between
$V$ and $V^{*}$.
 Moreover we write $\mathscr{L}(V,V^*)$ the space of bounded
linear transformations of V into $V^*$.

Now we recall an It\^o's formula in Hilbert space
which will be  frequently  used  in this paper.

\begin{lem}\label{lem:c1}
Let $\varphi\in L^{2}(\Omega,\mathscr{F}_{0},P; H)$. Let $Y, Z$ and $\Gamma$  be three
  progressively measurable stochastic processes defined on $[0,T]\times \Omega$ with values in $V,H$ and $V^{*}$ such that $ Y\in { M}_{\mathscr F}^2 (0, T; V),  Z\in {M}_{\mathscr F}^2 (0, T; H)$ and $\Gamma \in { M}_{\mathscr F}^2 (0, T; V^*) $, respectively. Let $R$ be a ${\mathscr P} \otimes {\mathscr B} ({E})$ -measurable stochastic process defined on $[0,T]\times
\Omega\times E$ with values in $H$ such that
 $R\in {M}_\mathscr{F}^{\nu,2}{([0,T]\times  E; H)}.$
 Suppose that
for every
  $\eta \in V$ and almost every $(\omega,t)\in\Omega\times[0,T]$, it holds
  that
  \begin{equation*}
    ( \eta,Y)_H =( \eta, \varphi)_H+
    \int_{0}^{t} \la \eta,\Gamma(s) \ra ds
    + \int_0^t( \eta, Z )_HdW(s)+\int_0^t
    \int_{E}( \eta, R(s,e) )_H\tilde \mu(de,ds) .
  \end{equation*}
  Then  $Y$ is a $H$-valued  strongly cadlag
  $\mathscr F_t$-adapted process such that
\begin{eqnarray} \label{eq:2.111}
  \mathbb E \bigg [\sup_{0\leq t\leq T} ||Y||_H^2\bigg ] < \infty
\end{eqnarray}
and the following It\^{o} formula holds

 \begin{eqnarray}\label{eq:2.1}
   \begin{split}
     ||Y(t)||_H^2=&||\varphi||^2+
     2\int_0^t\langle \Gamma(s), Y(s) \rangle
     ds +2\int_0^t ( Z(s), Y(s))_H
     dW(s)+ \int_0^t||Z(s)||_H^2ds
     \\&+\int_0^t\int_E\bigg[ ||R(s,e)||_H^2+
     2 (Y(s), R(s,e))_H   \bigg]
     \tilde \mu(de,ds)+\int_0^t\int_E ||R(s,e)||_H^2\nu(de)ds.
   \end{split}
 \end{eqnarray}
\end{lem}

\begin{proof}
  The proof follows that of Theorem 1 in Gy\"{o}ngy and Krylov [13].
\end{proof}

\section{Stochastic Evolution Equation with Jumps}

In this section, we study  the existence and uniqueness of the solution of the following  stochastic evolution equation (SEE, for short)
with jumps:

\begin{eqnarray} \label{eq:3.1}
  \left\{
  \begin{aligned}
   d X (t) = & \ [ A (t) X (t) + b ( t, X (t)) ] d t
+ [B(t)X(t)+g( t, X (t)) ]d W(t)
 \\&\quad +\int_E \sigma (t,e, X(t))\tilde \mu(de,dt),  \\
X (0) = & \  x \in H , \quad t \in [ 0, T ],
  \end{aligned}
  \right.
\end{eqnarray}
where $A,B,b,g$ and  $\sigma $ are given random mappings
which satisfying the following
standing  assumptions.

\begin{ass} \label{ass:3.1}
  (i) The operator processes $A:[0,T]\times \Omega \longrightarrow {\mathscr L} (V, V^*)$ and $B
  : [0,T]\times \Omega \longrightarrow {\mathscr L} (V, H)$
  are weakly predictable; i.e.,
  $ \langle A(\cdot)x, y \rangle$ and $(B(\cdot)x, y)_H$
  are both predictable process for every $x, y\in V, $
  and satisfy the coercive condition, i.e.,  there exists
  some constants  $ C, \alpha>0$ and $\lambda$ such that
  a.s.$(t,\omega)\in [0,T]\times \Omega$
  for all $x\in V,$
    \begin{eqnarray}
    \begin{split}
     - \langle A(t)x, x \rangle +\lambda ||x||_H \geq \alpha
      ||x||_V+||Bx||_H.
    \end{split}
  \end{eqnarray}
  and  \begin{eqnarray}
\sup_{( t, \omega ) \in [0, T] \times \Omega} \| A ( t,\omega ) \|_{{\mathscr L} ( V, V^* )}
 +\sup_{( t, \omega ) \in [0, T] \times \Omega} \| B ( t,\omega ) \|_{{\mathscr L} ( V, H )} \leq C \ .
\end{eqnarray}
\end{ass}

\begin{ass} \label{ass:3.2}
   The mappings $b:[0,T]\times  \Omega\times H  \longrightarrow H$ and $g:[0,T]\times\Omega\times H  \longrightarrow H$  are both $\mathscr P\times
   \mathscr B(H)/\mathscr B(H) $-measurable; the mapping $\sigma:[0,T]\times  \Omega \times
   E\times H  \longrightarrow H$ is $\mathscr P\times\mathscr B(E)  \times
   \mathscr B(H)/\mathscr B(H) $-measurable.
   And there exists a constant $C$  such that
   for all $x, \bar x\in V$ and a.s.$(t,\omega)\in [0,T]\times \Omega$
   \begin{eqnarray}
     \begin{split}
       ||b(t,x)-b(t,x)||_H+ ||g(t,x)-g(t,x)||_H
       +||\sigma(t,x,\cdot)-\sigma(t,x.\cdot)||
       _{M^{\nu,2}( Z; H)} \leq  C||x-\bar x||_H.
     \end{split}
   \end{eqnarray}

\end{ass}

\begin{defn}
\label{defn:c1}
A $V$-valued, $\{{\mathscr F}_t\}_{0\leq t\leq T}$-adapted process $X(\cdot)$ is said to be a solution to the
SEE \eqref{eq:3.1}, if $X (\cdot) \in { M}_{
\mathscr F}^2 ( 0, T; V )$, such that for every $\phi \in V$
and a.e. $( t, \omega ) \in [0, T ] \times \Omega$, it holds that
\begin{eqnarray}
\left\{
\begin{aligned}
( X (t), \phi )_H =& \ ( x, \phi )_H + \int_0^t \left < A (s) X (s), \phi \right > d s
+\int_0^t ( b ( s, X (s), \phi )_H d s \\
& + \int_0^t ( B(s)X(s)+ g ( s, X (s) ), \phi )_H d W (s)
\\& + \int_0^t \int_{E} (\sigma ( s,e, X (s) ), \phi )_H d \tilde \mu (de,ds), \quad t \in [ 0, T ] , \\
X(0) =& \ x\in H  ,
\end{aligned}
\right.
\end{eqnarray}
or alternatively, $X (\cdot)$ satisfies the following It\^o's equation in $V^*$:
\begin{eqnarray}
\left\{
\begin{aligned}
X (t)=& \  x+ \int_0^t  A (s) X (s)d s
+\int_0^t  b ( s, X (s)d s + \int_0^t 
[B(s)X(s)+ g ( s, X (s) )] d W (s)
\\& + \int_0^t \int_{E} \sigma ( s,e, X (s) ) d \tilde \mu (de,ds), \quad t \in [ 0, T ] , \\
X(t) =& \ x\in H .
\end{aligned}
\right.
\end{eqnarray}
\end{defn}

Now we state our main result.

\begin{thm} \label{thm:3.1}
  Let Assumptions \ref{ass:3.1}-\ref{ass:3.2} be
  satisfied by given coefficients $(A,B,b,g,\sigma)$.  Then  SEE \eqref{eq:3.1} has a unique
  solution $X(\cdot)\in M_{\mathscr{F}}^2(0,T;V) \bigcap  S_{\mathscr{F}}^2(0,T;H).$
\end{thm}
To prove this  theorem, we need the following
result on  the a prior estimate for the solution
 to   SEE \eqref{eq:3.1}.

 \begin{thm} \label{thm:3.2}
  Let  $ X(\cdot)$  be the solution to
  the  SEE   \eqref{eq:3.1}
   with the coefficients $(A,B, b,g,\sigma)$
   satisfying Assumptions \ref{ass:3.1}-\ref{ass:3.2}. Then
  the following estimate holds:
\begin{eqnarray}\label{eq:3.4}
\begin{split}
&{\mathbb E} \bigg [ \sup_{0 \leq t \leq T} \| X (t) \|_H^2 \bigg
]
+ {\mathbb E} \bigg [ \int_0^T \| X (t) \|_V^2 d t \bigg ] \\
& \leq K \bigg \{ ||x||_H + {\mathbb E}
\bigg [ \int_0^T \| b ( t, 0) \|_H^2 d t \bigg ] + {\mathbb E}
\bigg [ \int_0^T \| g ( t, 0) \|_H^2 d t \bigg ]
+ {\mathbb E} \bigg [ \int_{0}^T\int_E \| \sigma (t,e,0) \|^2_H  \nu(de)d t \bigg ] \bigg \}.
\end{split}
\end{eqnarray}
Furthermore, suppose that   $ \bar X(\cdot)$  is the solution to
  the  SEE   \eqref{eq:3.1}
   with the initial value $\bar X(0)=\bar x$ and  coefficients $(A,B,  \bar b, \bar g,\bar\sigma)$
   satisfying Assumptions \ref{ass:3.1}-\ref{ass:3.2},
   then we have

   \begin{eqnarray}\label{eq:3.5}
&& {\mathbb E} \bigg [ \sup_{0 \leq t \leq T} \| X (t) - {\bar X} (t) \|_H^2 \bigg ]
+ {\mathbb E} \bigg [ \int_0^T \| X (t) - {\bar X} (t) \|_V^2 d t \bigg ] \nonumber \\
&& \leq K \bigg \{  \|x-\bar x\|^2_H
+ {\mathbb E} \bigg [ \int_0^T \| b ( t, {\bar X} (t) )
- {\bar b} ( t, {\bar X} (t) ) \|_H^2 d t \bigg ] \\
&&+ {\mathbb E} \bigg [ \int_0^T \| g ( t, {\bar X} (t) ) - {\bar g}( t, {\bar X} (t) ) \|_H^2 d t \bigg ]
+ {\mathbb E} \bigg [ \int_0^T \int_{E}\| \sigma ( t, e, {\bar X} (t)) - {\bar \sigma}( t, e, {\bar X} (t) ) \|_H^2 \nu(de)d t \bigg ]   \bigg \} .\nonumber
\end{eqnarray}
 \end{thm}

 \begin{proof}
   The estimate  \eqref{eq:3.4}
   can be directly obtained by
   the estimate \eqref {eq:3.5}
   by taking  the initial value $\bar X(0)=0$ and  coefficients $(A,B,  \bar b, \bar g,\bar\sigma)=(A, B,0,0,0)$ which imply that
   $\bar X(\cdot)\equiv 0.$
   Therefore, it suffices to prove the estimate
   \eqref {eq:3.5}.
   For the sake of simplicity,
   in the following discussion, we will
   use the following shorthand notation:

   \begin{eqnarray*}
  && {\hat X} (t) \triangleq X (t) - {\bar X} (t) , \quad \hat x \triangleq x - \bar x,
  \\&& \Lambda\triangleq  \|x-\bar x\|^2_H
+ {\mathbb E} \bigg [ \int_0^T \| b ( t, {\bar X} (t) )
- {\bar b} ( t, {\bar X} (t) ) \|_H^2 d t \bigg ] + {\mathbb E} \bigg [ \int_0^T \| g ( t, {\bar X} (t) ) - {\bar g}( t, {\bar X} (t) ) \|_H^2 d t \bigg ]
\\&&\quad \quad \quad+ {\mathbb E} \bigg [ \int_0^T \int_{E}\| \sigma ( t, {\bar X} (t), e) - {\bar \sigma}( t, {\bar X} (t),e ) \|_H^2 \nu(de)d t \bigg ]
\\
&& {\tilde \phi} (t) \triangleq \phi ( t, { X} (t))
- {\bar \phi} ( t, {\bar X} (t) ),
{\hat \phi} (t) \triangleq \phi ( t, {\bar X} (t))
- {\bar \phi} ( t, {\bar X} (t) ) ,
 {\Delta \phi} (t) \triangleq \phi ( t, { X} (t))
- { \phi} ( t, {\bar X} (t) ),
\end{eqnarray*}
where $\phi=b,g,\sigma.$

Applying It\^{o} formula in Lemma \ref{lem:c1} to $||\hat X(t)||_H^2$
and using Assumption \ref{ass:3.1}-\ref{ass:3.2} and the elementary inequalities $|a+b|^2\leq 2a^2+2 b^2$
and $2
a b \leq \frac{1}{\epsilon} a^2 + \epsilon b^2$, $\forall a, b > 0$,
$\epsilon > 0$ and
, we get  that

\begin{eqnarray}\label{eq:3.6}
     ||\hat X(t)||_H^2=&&||\hat x||^2_H+
     2\int_0^t\langle  A(s)\hat X(s), \hat X(s) \rangle
     ds +2\int_0^t  ( \hat X(s), \tilde b(s))_H ds
+\int_0^t||B(s)\hat X(s)+\tilde g(s)||_H^2ds\nonumber
  \\&&+\int_0^t\int_E ||\tilde \sigma (s,e)||_H^2\nu(de)ds
+2\int_0^t  ( \hat X(s), B(s)\hat X(s)+
     \tilde g(s))_H
     dW(s)\nonumber
     \\&&+\int_0^t\int_E\bigg[ ||\tilde \sigma(s,e)||_H^2+
     2 (\hat X(s), \tilde \sigma(s,e))   \bigg]
     \tilde \mu(de,ds)\nonumber
     \\=&&||\hat x||^2_H+
     2\int_0^t\bigg[\langle  A(s)\hat X(s), \hat X(s) \rangle+||B(s)\hat X(s)||_H^2\bigg]
     ds+\int_0^t ||\Delta g(s)+\hat g(s)||_H^2\nonumber
     \\&&
     +2\int_0^t  (  B(s)\hat X(s), \Delta g(s)+\hat g(s))_H
     ds+2\int_0^t  ( \hat X(s), \Delta b(s)+\hat b(s))_H
     ds\nonumber
\\&&+\int_0^t\int_E ||\Delta\sigma(s,e)+\hat \sigma (s,e)||_H^2 \nu(de)ds+
     2\int_0^t  ( \hat X(s), B(s)\hat X(s)
     +\tilde g(s))_H dW(s)\nonumber
     \\&&+\int_0^t\int_E\bigg[ ||\tilde \sigma(s,e)||_H^2+
     2 (\hat X(s), \tilde \sigma(s,e))   \bigg]
     \tilde \mu(de,ds)\nonumber
     \\ \leq &&
     K\Lambda+ (-2\alpha+\varepsilon){\mathbb E} \bigg [
\int_0^t   \| \hat x ( s) \|_V^2ds\bigg] +K{\mathbb E} \bigg [
\int_0^t \| \hat X ( s) \|_H^2dt\bigg]\nonumber
\\&&+  2\int_0^t  ( \hat X(s), B(s)\hat X(s)
     +\tilde g(s))_H dW(s)
\\&&+\int_0^t\int_E\bigg[ ||\tilde \sigma(s,e)||_H^2+
     2 (\hat X(s), \tilde \sigma(s,e))   \bigg]
     \tilde \mu(de,ds).
 \end{eqnarray}
Taking expectation on both sides of the above
and taking  $\varepsilon$ small enough such that
$-2\alpha+\varepsilon <0,$  we  conclude that

\begin{eqnarray}\label{eq:3.7}
\begin{split}
    \mathbb E[ ||\hat X(t)||_H^2]
    +\mathbb E\bigg[\int_0^t ||\hat X(s)||_V^2ds\bigg]
    \leq
     K\Lambda
+K{\mathbb E} \bigg [
\int_0^t \| \hat X ( s) \|_H^2dt\bigg].
\end{split}
 \end{eqnarray}
Then by virtue of  Gr\"onwall's inequality to  $\mathbb E[||X(t)||_H^2],$ we obtain
\begin{eqnarray}\label{eq:3.8}
\begin{split}
    \sup_{0\leq t\leq T}\mathbb E[ ||\hat X(t)||_H^2]
    +\mathbb E\bigg[\int_0^T ||\hat X(s)||_V^2ds\bigg]
    \leq &
     K \Lambda.
\end{split}
 \end{eqnarray}
Furthermore, applying Burkholder-Davis-Gundy inequality in \eqref{eq:3.6}
and using  the estimate \eqref{eq:3.7},
we get that
\begin{eqnarray}
  \begin{split}
   {\mathbb E} \bigg [ \sup_{0 \leq t \leq T} \|
   \hat X (t) \|_H^2 \bigg ]
   \leq& K \Lambda+ 2{\mathbb E} \bigg \{ \sup_{0 \leq t \leq T}\bigg |\int_0^t( \hat X(s), B(s)\hat X(s)+\tilde g(s))_H dW(s)\bigg|\bigg\}
   \\&+ 2{\mathbb E}\bigg \{ \sup_{0 \leq t \leq T}\bigg |\int_0^t\int_{E}\bigg[ ||\tilde \sigma(s,e)||_H^2+
     2 (\hat X(s), \tilde \sigma(s,e))\bigg]\tilde \mu(de,dt)\bigg|\bigg\}
    \\ \leq& K \Lambda+ K{\mathbb E} \bigg \{
    \int_0^T \bigg |( \hat X(s), B(s)\hat X(s)+\tilde g(s))_H\bigg|^2ds\bigg\}^{\frac{1}{2}}
   \\&+ K{\mathbb E} \bigg \{
   \int_0^T\int_{E} \bigg | ||\tilde \sigma(s,e)||_H^2+
     2 (\hat X(s), \tilde \sigma(s,e))\bigg|
     \nu(de)dt\bigg\}
     \\&\leq  \frac{1}{2}{\mathbb E} \bigg [ \sup_{0 \leq t \leq T} \|
   \hat X (t) \|_H^2 \bigg ] +K\Lambda,
  \end{split}
\end{eqnarray}
which implies that

\begin{eqnarray}\label{eq:3.10}
  {\mathbb E} \bigg [ \sup_{0 \leq t \leq T} \|
   \hat X (t) \|_H^2 \bigg ]\leq K\Lambda.
\end{eqnarray}
Combining \eqref{eq:3.8} and
\eqref{eq:3.10}, we get the desired result.
The proof is complete.
 \end{proof}
Now we give the existence and
uniqueness of the solution
of \eqref{eq:3.1}
for a simple case where the
coefficients $(b, g, \sigma)$ is
independent of the variable $x.$

\begin{lem} \label{lem:3.3}
  Given three stochastic processes  $b,g$ and $\sigma$  such that $b\in M_{\mathscr{F}}^2(0,T;H), g\in M_{\mathscr{F}}^2(0,T;H)$ and  $\sigma\in M_{\mathscr{F}}^{\nu,2}(0,T;H).$
 Suppose that the operators $ A$ and $B$  satisfy  Assumption \ref{ass:3.1}.
  There exists a unique solution to the
  following SEE:
\begin{eqnarray} \label{eq:3.12}
  \left\{
  \begin{aligned}
   d X (t) = & \ [ A (t) X (t) + b ( t) ] d t
+ [B(t)X(t)+g( t) ]d W(t)+\int_E \sigma (t,e)\tilde \mu(de,dt),  \\
X (0) = & \  x , \quad t \in [ 0, T ].
  \end{aligned}
  \right.
\end{eqnarray}
\end{lem}

\begin{proof}
  The  proof can be obtained by Galerkin approximations  in the same way as
  the proof of  Theorem 3.2 in [11]
  with minor change.
\end{proof}

{ \bf\emph{Proof of Theorem \ref{thm:3.1}.}} The uniqueness of the solution to the SEE
\eqref{eq:3.1} can be got   by the a priori estimate
\eqref{eq:3.5} directly.
  For $\rho \in [0, 1] $ and  three   any  given stochastic processes
   $b_0 (\cdot) \in { M}_{\mathscr F}^2 ( 0, T; H ),$
$g_0 (\cdot) \in {M}_{\mathscr F}^2 ( 0, T; H ),$  and  $\sigma_0(\cdot)\in M_{\mathscr{F}}^{\nu,2}(0,T;H),$ we introduce  a family of parameterized  SEEs as follows:
\begin{eqnarray}\label{eq:3.12}
\begin{split}
 X(t) =& \ x +\int_0^t A (s) X (s) d s + \int_0^t \Big [\rho  b ( s, X(s) )]+ b_0(t)\Big]d s+\int_0^t \Big[ B(s)X(s)+\rho g ( s, X ( s ) )+g_0(t)\Big ] d W (s)
 \\&+\int_E \Big[\rho \sigma (t,e, X(t))
 +\sigma_0(t,e)\Big]\tilde \mu(de,dt).
\end{split}
\end{eqnarray}
It is easy to see that
when we take the parameter $\rho=1$ and
$b_0 (\cdot)\equiv 0,g_0 (\cdot)\equiv0,
\sigma_0 (\cdot, \cdot)\equiv0,$
 the  SEE \eqref{eq:3.12}
 is reduced to  the original SEE
\eqref{eq:3.1}. Obviously,  the  coefficients of the SEE \eqref{eq:3.12}
satisfy  Assumption \ref{ass:3.1} and
\ref{ass:3.2}
with $(A, B,  b ,  g, \sigma)$  replaced by $(A, B, \rho b + b_0, \rho g+ g_0, \rho \sigma+ \sigma_0)$.
 Suppose
 for any
$b_0 (\cdot) \in { M}_{\mathscr F}^2 ( 0, T; H ),$
$g_0 (\cdot) \in { M}_{\mathscr F}^2 ( 0, T; H ),$  $\sigma_0(\cdot)\in M_{\mathscr{F}}^{\nu,2}(0,T;H),$ and  some  parameter $\rho = \rho_0$,
 there exists a unique solution $X (\cdot)\in { M}_{\mathscr F}^2 ( 0, T; V)$ to the SEE  \eqref{eq:3.12}.  For any parameter $\rho$,
the SEE \eqref{eq:3.12}
can be rewritten as
\begin{eqnarray}\label{eq:3.14}
\begin{split}
 X(t) =& \ x + \int_0^t A (s) X (s) d s + \int_0^t \Big [\rho_0 b ( s,X(s) )+ b_0(t)+(\rho-\rho_0)  b ( s,  X(s) )\Big]d s
 \\&+\int_0^t \Big[ B(s)X(s)+\rho_0g ( s, X ( s ) )+g_0(t)
 +(\rho-\rho_0)g ( s, X ( s ) )\Big ] d W (s)
 \\&+\int_0^t\int_E \Big[\rho_0
  \sigma( s,e, X ( s ) )+\sigma_0(t,e)
 +(\rho-\rho_0)\sigma( s,e, X ( s ) )\Big ] d \tilde \mu(de,ds).
\end{split}
\end{eqnarray}
 Therefore, by the above assumption,
  for any $x (\cdot) \in {M}_{\mathscr F}^2 ( 0, T; V),$
the  SEE
\begin{eqnarray}\label{eq:3.19}
\begin{split}
 X(t) =& \ x + \int_0^t A (s) X (s) d s + \int_0^t \Big [\rho_0 b ( s,X(s) )+ b_0(t)+(\rho-\rho_0)  b ( s,  x(s) )\Big]d s
 \\&+\int_0^t \Big[ B(s)X(s)+\rho_0g ( s, X ( s ) )+g_0(t)
 +(\rho-\rho_0)g ( s, x ( s ) )\Big ] d W (s)
 \\&+\int_0^t\int_E \Big[\rho_0
  \sigma( s,e, X ( s ) )+\sigma_0(t,e)
 +(\rho-\rho_0)\sigma( s,e, x ( s ) )\Big ] d \tilde \mu(de,ds)
\end{split}
\end{eqnarray}
 admits a unique solution $X (\cdot) \in { M}_{\mathscr F}^2 ( 0, T; V)$.
Now define  a mapping from $ { M}_{\mathscr F}^2 ( 0, T; V)$ onto itself denoted by $$X (\cdot) = \Gamma (x (\cdot)).$$
Then for any $x_i (\cdot) \in { M}_{\mathscr F}^2 ( 0, T; V)$, $i = 1, 2$, from the Lipschitz
continuity of $b$, $g$, $\sigma$ and  a priori estimate \eqref{eq:3.4},
it follows that
\begin{eqnarray*}
|| \Gamma ( x_1 (\cdot) ) - \Gamma ( x_2 (\cdot) ) ||^2_{{ M}_{\mathscr F}^2 ( 0, T; V)}
&=& || X_1 (\cdot) - X_2(\cdot) ||^2_{{ M}_{\mathscr F}^2 ( 0, T; V)}\\
&\leq& K |\rho-\rho_0|^2 \cdot || x_1 (\cdot) - x_2(\cdot) ||^2_{{ M}_{\mathscr F}^2 ( 0, T; V)}.
\end{eqnarray*}
Here $K $ is a  positive constant independent of $\rho$. If $| \rho - \rho_0 |< \frac{1}{2 \sqrt {K}}$, the mapping $\Gamma$ is strictly contractive in
${ M}_{\cal F}^2 ( 0, T; V )$. Hence it implies that
the SEE \eqref{eq:3.12} with the coefficients $( A, B, \rho b + b_0, \rho g + g_0, \rho \sigma + \sigma_0, )$ admits a unique solution
$X (\cdot)\in {M}_{\cal F}^2 ( 0, T; V )$.
From Lemma \ref{lem:3.3},  the uniqueness and existence of a solution to the SEE \eqref{eq:3.12} is true for $\rho=0$. Then starting from $\rho = 0$,  one can reach $\rho = 1$ in finite steps and this finishes
 the proof of solvability of the SEE \eqref{eq:3.1}.
Moreover,  from Lemma \ref{lem:c1} and  the  estimate \eqref{eq:2.111},
we obtain $ X(\cdot) \in {\cal S}_{\mathscr F}^2 ( 0, T; H )$. This completes the proof.

\section{ Formulation of  Optimal Control Problem }
Let $U$  be a real-valued Hilbert space standing for the control space. Let  ${\mathscr U}$  be a nonempty  convex
closed subset of $U$.  An  admissible control process $u (\cdot) \triangleq \{ u (t), 0 \leq t \leq T \}$ is  defined
as follows.
\begin{defn}
A stochastic  process $u (\cdot)$
 defined on $[0, T]\times \Omega$ is called an admissible control process  if it is a predictable  process such that
$u (\cdot) \in { M}^2 ( 0, T; U ).$
The set of all admissible controls is
denoted by  ${\cal A}$.
\end{defn}
We make the following
basic assumptions.
\begin{ass}\label{ass:2.5}
\begin{enumerate}
\item[]
\item[(i)]

$A:[0,T]\times \Omega \longrightarrow {\mathscr L} (V, V^*)$
and
$B: [0,T]\times \Omega \longrightarrow {\mathscr L} (V, H)$
are operator-valued stochastic process satisfying (i) in Assumption \ref{ass:3.1};
\item[(iii)]$b, g: [ 0, T ] \times \Omega \times H \times {\mathscr U} \rightarrow H$  are $\mathscr P\times
   \mathscr B(H)\times \mathscr B(\mathscr U)/\mathscr B(H) $ measurable
   mappings and   $\sigma:[0,T]\times  \Omega \times
   E\times H \times \mathscr U \longrightarrow H$ is a $\mathscr P\times \mathscr B(E)\times
   \mathscr B(H)\times \mathscr B(U)/\mathscr B(H)$-measurable mapping  such that $b ( \cdot, 0, 0 ), g ( \cdot, 0, 0 ) \in {
M}^2_{\mathscr F} ( 0, T; H ), \sigma(\cdot, \cdot, 0,0)\in {M}_\mathscr{F}^{\nu,2}{([0,T]\times  E; H)}.$ Moreover, for almost all $( t, \omega, e ) \in [ 0, T ] \times \Omega \times E$,  $h$, $g$ and $\sigma$ are  G\^ateaux differentiable in $(x,u)$ with  continuous bounded G\^ateaux  derivatives
$b_x, g_x,\sigma_x,  b_u, g_u$ and  $\sigma_u$;
\item[(iv)]
$l:[ 0, T ] \times \Omega \times H \times {\mathscr U} \rightarrow H $ is a
 ${\mathscr P} \otimes {\mathscr B} (H) \otimes {\mathscr B} ({\mathscr U})/ {\mathscr B} ({\mathbb R})$-measurable mapping  and  $\Phi: \Omega \times H \rightarrow {\mathbb R}$
is  a ${\mathscr F}_T\otimes {\mathscr B} (H) / {\mathscr B} ({\mathbb R})$-measurable
mapping.
For almost all $( t, \omega ) \in [ 0, T ] \times \Omega$, $l$ is continuous G\^ateaux
differentiable in $(x,u)$
with continuous  G\^ateaux derivatives $l_x$ and $l_u$, $\Phi$
is G\^ateaux differentiable with  continuous
 G\^ateaux derivative $\Phi_x$.
Moreover, for almost all $( t, \omega ) \in [ 0, T ] \times \Omega$,   there exists a   constant $C > 0$ such that  for all $( x, u ) \in H  \times {\mathscr U}$
\begin{eqnarray*}
| l ( t, x, u  ) |
\leq  C ( 1 + \| x \|^2_H + + \| u \|_U^2 ) ,
\end{eqnarray*}
\begin{eqnarray*}
&& \| l_x ( t, x,u) \|_H +
+ \| l_u ( t, x, u ) \|_U \leq C ( 1 + \| x \|_H  + \| u \|_U  ) ,
\end{eqnarray*}
and
\begin{eqnarray*}
& | \Phi (x) | \leq C ( 1 + \| x \|^2_H) , \\
& \| \Phi_x (x) \|_H \leq C ( 1 + \| x \|_H).
\end{eqnarray*}
\end{enumerate}
\end{ass}
In the Gelfand triple $( V, H, V^* ),$
for any admissible control $u(\cdot)\in \cal A,$
we consider  the following
SEE
\begin{eqnarray} \label{eq:4.1}
  \left\{
  \begin{aligned}
   d X (t) = & \ [ A (t) X (t) + b ( t, X (t), u(t)) ] d t
+ [B(t)X(t)+g( t, X (t), u(t)) ]d W(t)
 \\&\quad +\int_E \sigma (t, e,X(t),u(t))\tilde \mu(de,dt),  \\
X (0) = & \  x , \quad t \in [ 0, T ]
  \end{aligned}
  \right.
\end{eqnarray}
with the cost functional
\begin{eqnarray}\label{eq:4.2}
J ( u (\cdot) ) = {\mathbb E} \bigg [ \int_0^T l ( t, x (t), u (t) ) d t
+ \Phi ( x (T) ) \bigg ].
\end{eqnarray}
For any admissible control $u(\cdot),$ the solution of the system { \eqref{eq:4.1}},  denoted by $X^u(\cdot)$ or $X(\cdot),$
  if its dependence on
  admissible control $u(\cdot)$ is clear from  the context,  is called the
state process corresponding to the control process $u(\cdot)$, and
 $(u(\cdot), X(\cdot))$ is called an
admissible pair.

The following result gives the well-posedness of the state equation as well as some useful estimates.

\begin{lem}\label{lem:1.1}
Let Assumption \ref{ass:2.5}  be satisfied. Then for any admissible control $u(\cdot)$,
the state equation
\eqref{eq:4.1}
has a unique solution $X^u(\cdot) \in M_{\mathscr{F}}^2(0,T;V) \bigcap  S_{\mathscr{F}}^2(0,T;H).$
Moreover,  the following estimate holds

\begin{eqnarray}\label{eq:4.3}
\begin{split}
&{\mathbb E} \bigg [ \sup_{0 \leq t \leq T} \|
 X^u (t) \|_H^2 \bigg
]
+ {\mathbb E} \bigg [ \int_0^T \| X^u (t) \|_V^2 d t \bigg ] \leq K \bigg \{ 1+||x||_H + {\mathbb E}
\bigg [ \int_0^T \| u ( t) \|_U^2 d t \bigg ]\bigg \}
\end{split}
\end{eqnarray}
and
\begin{eqnarray}\label{eq:4.4}
\begin{split}
  |J( u(\cdot))|< \infty.
  \end{split}
\end{eqnarray}
Furthermore, let $ X^v(\cdot)$  be
the state process corresponding to
another admissible control $v(\cdot),$
then

   \begin{eqnarray}\label{eq:4.5}
&& {\mathbb E} \bigg [ \sup_{0 \leq t \leq T} \| X^u (t) - { X}^v (t) \|_H^2 \bigg ]
+ {\mathbb E}
\bigg [ \int_0^T \| X^u (t) - { X}^v (t) \|_V^2 d t \bigg ] \leq K {\mathbb E} \bigg [ \int_0^T \| u(t) -v(t) \|_U^2 d t \bigg ] .
\end{eqnarray}
\end{lem}

\begin{proof}
  Under Assumption \ref{ass:2.5}, by Theorem \ref{thm:3.1}, we  can get directly
  the existence and uniqueness of the  solution of the state equation \eqref{eq:3.1},  and
  the estimates \eqref{eq:4.3} and \eqref{eq:4.5} can be obtained by
  the estimates \eqref{eq:3.5} and \eqref{eq:3.6}, respectively.
  Furthermore, from Assumption
   \ref{ass:2.5} and  the  estimate  \eqref{eq:4.3}, it follows that

   \begin{eqnarray}\label{eq:1.7}
     \begin{split}
       |J(x, u(\cdot))|\leq K \bigg\{\mathbb E \bigg[\sup_{0\leq t\leq T}|| X(t)||_H^2 \bigg] +\mathbb E\bigg[\int_0^T
  ||u(t)||^2_U dt\bigg]+1\bigg]\bigg\} \leq K \bigg\{\mathbb E\bigg[\int_0^T
  ||u(t)||_U^2 dt\bigg]+||x||_H+1\bigg]\bigg\}<\infty.
     \end{split}
   \end{eqnarray}
The proof is complete.
\end{proof}
Therefore, by Lemma \ref{lem:1.1}, we claim  that the cost functional \eqref{eq:4.2}
is well-defined.
Now we put forward our optimal control problem
as follows.

\begin{pro}\label{pro:2.1}
Find an admissible control process ${\bar u} (\cdot) \in {\cal A}$ such that
\begin{eqnarray}\label{eq:b8}
J ( {\bar u}(\cdot) ) = \inf_{u (\cdot) \in {\cal A}} J ( u (\cdot) ) .
\end{eqnarray}
\end{pro}
The admissible control ${\bar u} (\cdot)$ satisfying (\ref{eq:b8}) is called an optimal control process of
Problem \ref{pro:2.1}. Correspondingly, the state process ${\bar X} (\cdot)$ associated with ${\bar u} (\cdot)$
is called an optimal state process. Then $( {\bar u} (\cdot); {\bar X} (\cdot) )$ is called an optimal pair of
Problem \ref{pro:2.1}.

\section{Regularity Result for the Adjoint Equation}

For any admissible pair $( \bar u (\cdot); \bar X(\cdot) )$, the corresponding  adjoint
processes is defined as  the triple of processes $(p(\cdot), q(\cdot), r(\cdot, \cdot)),$ which is
the solution to
the following
backward stochastic equation  with jump,  called
adjoint equation,
\begin{eqnarray}\label{eq:4.4}
\begin{split}
   \left\{\begin{array}{ll}
dp(t)=&-\bigg[A^*(t)p(t)+b_x^*(t, \bar X(t), \bar u(t))p(t)
+B^*(t)q(t)+g_x^{*}(t, \bar X(t), \bar u(t))q(t)
\\&\quad\quad+\displaystyle\int_{{E}}\sigma_x^{*}(t,e,\bar
X(t),\bar u(t))r(t, e)\nu (de)dt
+l_x(t, \bar X(t), \bar u(t))\bigg]dt
\\&\quad\quad+q(t)dW(t)+\displaystyle
\int_{{E}}r(t, e)\tilde{\mu}(de, dt),
~~~~0\leqslant t\leqslant T,
\\p(T)=&\Phi_x(T),
  \end{array}
 \right.
 \end{split}
  \end{eqnarray}
Here $A^*$ denotes the adjoint operator of
the operator $A.$ Similarly, we can define
the corresponding adjoint operator for other operators.\\
Under Assumptions \ref{ass:2.5}, we have the following
basic result on the adjoint processes.

\begin{lem}
 Let Assumptions \ref{ass:2.5}
 be satisfied. Then for any admissible pair
 $( \bar u (\cdot); \bar X (\cdot) ),$
  there exists a unique adjoint processes
  $(p(\cdot), q(\cdot), r(\cdot))\in M_{\mathscr{F}}^2(0,T;H)\times M_{\mathscr{F}}^2(0,T;H)\times M_{\mathscr{F}}^{\nu,2}([0,T]\times E;H).$
  Moreover, the following estimate holds:

  \begin{eqnarray}
    \begin{split}
    & {\mathbb E} \bigg [ \sup_{0 \leq t \leq T} \| p (t) \|_H^2 \bigg
]+ {\mathbb E} \bigg [ \int_0^T \| p (t) \|_V^2 d t \bigg ] +{\mathbb E} \bigg [ \int_0^T \| q (t) \|_H^2 d t \bigg ]
+ {\mathbb E} \bigg [ \int_0^T
\int_{E}\| r (t,e) \|_H^2 \nu(de)d t \bigg ]
\\&\leq  K\bigg\{  {\mathbb E} \bigg [ \int_0^T \| l_x (t) \|_H^2 d t \bigg ]+
\mathbb E[||\Phi_x(T)||_H^2]\bigg\}
    \end{split}
  \end{eqnarray}
\end{lem}

\begin{proof}
From the property of adjoint operator,
 the adjoint operator $A^*$  of $A$  and the
  adjoint operator $B^*$ of $B$
 also  satisfies (i) in Assumption \ref{ass:3.1}. Therefore,
  the desired result can be obtained
  by the existence and uniqueness  theorem
  of solution of BSEE with jumps established in [19].
\end{proof}
Define the Hamiltonian ${\cal H}: [ 0, T ] \times \Omega \times H  \times {\mathscr U} \times V \times H \times  M^{\nu,2}( E; H)
\rightarrow {\mathbb R}$ by
\begin{eqnarray}\label{eq:5.3}
{\cal H} ( t, x, u, p, q, r ) := \left ( b ( t, x, u ), p \right )_H
+\left( g ( t, x,  u), q \right)_H
+\int_{E}\left( \sigma ( t,e, x,  u), r(t,e) \right)_H\nu(de)
+ l ( t, x, u ) .
\end{eqnarray}
Using Hamiltonian ${\cal H}$,
the adjoint equation \eqref{eq:4.4}
can be written in the following form:

\begin{eqnarray}\label{eq:5.4}
\begin{split}
   \left\{\begin{array}{ll}
dp(t)=&-\bigg[A^*(t)p(t)+B(t)^*q(t)+\bar {\cal H}_{x} (t)\bigg]dt+q(t)dW(t)+\displaystyle
\int_{{E}}r(t, e)\tilde{\mu}(de, dt),
~~~~0\leqslant t\leqslant T,
\\p(T)=&\Phi_x( \bar X(T)),
  \end{array}
 \right.
 \end{split}
  \end{eqnarray}
where we write
\begin{eqnarray}\label{eq:5.6}
\bar {\cal H} (t) \triangleq {\cal H} ( t,
\bar x (t), \bar u (t), p (t), q (t), r(t,\cdot) ).
\end{eqnarray}

\section{  Stochastic Maximum Principle}
\subsection {Variation of the State and Cost Functional}
Let$( {\bar u} (\cdot); {\bar X} (\cdot) )$ be an optimal pair.  Define  a
convex perturbation  of $\bar u(\cdot)$ as follows:
\begin{eqnarray*}
u^\epsilon ( \cdot ) \triangleq {\bar u} ( \cdot ) + \epsilon ( v ( \cdot ) - {\bar u} ( \cdot) ) ,
\quad 0 \leq \epsilon \leq 1,
\end{eqnarray*}
where $v(\cdot)$ is an arbitrarily
admissible control. Since the control domain $\mathscr U$ is
convex,  $u^\eps(\cdot)$ is also an element of
$\cal A.$
 We denote by $ X^\eps(\cdot)$  the state process
corresponding to the control $u^\eps(\cdot).$
Now we introduce the following  first order
variation equation:
\begin{eqnarray}\label{eq:5.8}
\left\{
\begin{aligned}
d Y (t) =& \ [ A (t) Y(t) + b_x ( t, \bar X(t),
\bar u(t) )Y(t) + b_u ( t,\bar X(t), \bar u(t) )(v(t)-\bar u(t)) ] d t
\\&+[B(t) Y(t)
+  g_x ( t,\bar X(t), \bar u(t))Y(t) + g_u ( t,
\bar X(t), \bar u(t) )(v(t)-\bar u(t))] d W (t)
\\& +\int_E\bigg[\sigma_x ( t,e,\bar X(t), \bar u(t) )Y(t)
 +\sigma_u ( t,e,\bar X(t), \bar u(t) )(v(t)-\bar u(t))\bigg]
 \tilde \mu(de,dt), \\
Y(0) =& 0.
\end{aligned}
\right.
\end{eqnarray}
Under Assumption \ref{ass:2.5}, by
Theorem \ref{thm:3.1}, we see that
the variation equation
\eqref{eq:5.8} has a unique
solution  $Y(\cdot)\in M_{\mathscr{F}}^2(0,T;V) \bigcap  S_{\mathscr{F}}^2(0,T;H).$

\begin{lem}\label{lem:6.1}
    Let  Assumption \ref{ass:2.5} be satisfied. Then we have the following estimates:
   \begin{eqnarray}\label{eq:6.2}
{\mathbb E} \bigg [ \sup_{0 \leq t \leq T} \| X^\epsilon (t) - {\bar X} (t) \|^2_H \bigg ]
+ {\mathbb E} \bigg [ \int_0^T \| X^\epsilon (t) - {\bar X} (t) \|^2_V d t \bigg ] = O (\epsilon^2) \ ,
\end{eqnarray}

  \begin{eqnarray}\label{eq:5.5}
{\mathbb E} \bigg [ \sup_{0 \leq t \leq T} \| X^\epsilon (t) - {\bar X} (t)-\eps Y(t) \|^2_H \bigg ]
+ {\mathbb E} \bigg [ \int_0^T \| X^\epsilon (t) - {\bar X} (t)-\eps Y(t) \|^2_V d t \bigg ] = o (\epsilon^2) \ .
\end{eqnarray}
 \end{lem}

 \begin{proof}
   From the estimate \eqref{eq:4.5},
   we have

   \begin{eqnarray}
   \begin{split}
{\mathbb E} \bigg [ \sup_{0 \leq t \leq T} \| X^\eps (t) - { \bar X} (t) \|_H^2 \bigg ]
+ {\mathbb E}
\bigg [ \int_0^T \| X^\eps (t) - { \bar X} (t) \|_V^2 d t \bigg ]
 &\leq K {\mathbb E} \bigg [ \int_0^T \| u^\eps(t) -\bar u(t) \|_U^2 d t \bigg ]
 \\&=K\eps^2 {\mathbb E} \bigg [ \int_0^T \| v(t) -\bar u(t) \|_U^2 d t \bigg ]
 \\&=O(\eps^2).
 \end{split}.
\end{eqnarray}
Set $\Xi^\eps(t)= X^\eps(t)-\bar X(t)-\eps Y(t).$
From   Taylor expanding  , we have
\begin{eqnarray} \label{eq:6.5}
  \left\{
  \begin{aligned}
   d \Xi^\eps (t) = & \ [ A (t) \Xi^\eps (t)
   +b_x(t, \bar X(t), \bar u(t))\Xi^\eps (t)
   + \alpha^\eps(t)]dt
   + [B(t)\Xi^\eps(t)+g_x(t, \bar X(t), \bar u(t))\Xi^\eps (t)
   + \beta^\eps(t) ]d W(t)\\
   &+\int_{E}\bigg[\sigma_x(t,e, \bar X(t), \bar u(t))\Xi^\eps (t)
   + \gamma^\eps(t) \bigg]d \tilde\mu(de,dt),  \\
X (0) = & \  x , \quad t \in [ 0, T ],
  \end{aligned}
  \right.
\end{eqnarray}
 \end{proof}
 where
 \begin{eqnarray} \label{eq:5.16}
\left\{
\begin{aligned}
\alpha^\eps(t)= &\int_0^1\bigg[\big(b_x( t, \bar X (t)+\lambda(X^\eps(t)-\bar X(t)), \bar u (t)+\lambda(u^\eps(t)-\bar u(t)))-b_x(t, \bar X(t), \bar u(t))\big)(X^\eps(t)-\bar X(t))
   \\&+\big(b_u( t,\bar X (t)+\lambda(X^\eps(t)-\bar X(t)),\bar u (t)+\lambda(u^\eps(t)-\bar u(t)))-b_u(t, \bar X(t), \bar u(t))\big)(u^\eps(t)-\bar u(t))\bigg]d\lambda,
   \\
   \beta^\eps(t)= &\int_0^1\bigg[\big(g_x( t, \bar X (t)+\lambda(X^\eps(t)-\bar X(t)), \bar u (t)+\lambda(u^\eps(t)-\bar u))-g_x(t, \bar X(t), \bar u(t))\big)(X^\eps(t)-\bar X(t))
   \\&+\big(g_u( t,\bar X (t)+\lambda(X^\eps(t)-\bar X(t)),\bar u (t)+\lambda(u^\eps(t)-\bar X))-g_u(t, \bar X(t), \bar u(t))\big)(u^\eps(t)-\bar u(t))\bigg]d\lambda,\\
   \gamma^\eps(t,e)= &\int_0^1\bigg[\big(\sigma_x( t, e, \bar X (t)+\lambda(X^\eps(t)-\bar X(t)), \bar u (t)+\lambda(u^\eps(t)-\bar u(t))-
   \sigma_x(t,e, \bar X(t), \bar u(t))\big)(X^\eps(t)-\bar X(t))
   \\&+\big(\sigma_u( t,e,\bar X (t)+\lambda(X^\eps(t)-\bar X),\bar u (t)+\lambda(u^\eps-\bar u(t)))-
   \sigma_x(t, e,\bar X(t), \bar u(t))\big)(u^\eps(t)-\bar u(t))\bigg]d\lambda.
\end{aligned}
\right.
\end{eqnarray}
 From the estimates \eqref{eq:3.4},
 \eqref{eq:6.2} and  Lebesgue dominated convergence theorem, we get that

 \begin{eqnarray}
&&{\mathbb E} \bigg [ \sup_{0 \leq t \leq T} \| \Xi(t)\|^2_H \bigg ]
+ {\mathbb E} \bigg [ \int_0^T \| \Xi(t) \|^2_V d t \bigg ]
\\&\leq& \mathbb E\bigg[ \int_0^T || \alpha^\eps(t)||^2_Hdt  \bigg]
+\mathbb E\bigg[\int_0^T | |\beta^\eps(t)||^2_Hdt\bigg]
 +\mathbb E\bigg[ \int_0^T \int_{E}
 ||\gamma^\eps(t,e)||^2_H\nu (de)dt \bigg] \nonumber \\
&=& \quad o(\eps).
\end{eqnarray}

 \begin{lem}
    Let  Assumption \ref{ass:2.5} be satisfied.                Let $( {\bar u} (\cdot); {\bar X} (\cdot) )$ be
 an optimal pair of
 Problem \ref{pro:2.1} associated with
  the first order variation process $Y(\cdot).$ Then,

  \begin{eqnarray}\label{eq:6.9}
  \begin{split}
    J(u^\eps(\cdot))-J(\bar u(\cdot))=&
    \eps\mathbb E\bigg[(\Phi_x( \bar X(T)), Y(T))_H\bigg]
      + \eps\mathbb E\bigg[\int_0^T
      ( l_x(t,\bar X(t), \bar u(t) ), Y(t))_Hdt\bigg]
      \\&+ \eps\mathbb E\bigg[\int_0^T
      ( l_u(t, \bar X(t), \bar u(t)), v(t)-u(t))_Udt\bigg]+o(\eps)
  \end{split}
\end{eqnarray}
 \end{lem}

\begin{proof}

From the definition of the cost functional, we have
\begin{eqnarray} \label{eq:6.10}
J(u^\eps(\cdot))- J(\bar u(\cdot))
={\mathbb E} \bigg [ \int_0^T
\Big(l ( t, X^\eps (t), u^\eps (t) ) - l ( t, \bar X(t), \bar u (t) \Big) d t\bigg ]
+ \mathbb E\bigg[ \Phi ( X^\eps (T) )
-\Phi ( \bar X (T) ) \bigg ]=I_1+I_2 \end{eqnarray}

Let us concentrate on $I_1, $
in terms of Lemma \ref{lem:6.1} and
the control convergence theorem, we have

\begin{eqnarray} \label{eq:6.11}
  \begin{split}
    I_1= &\mathbb E\bigg[\int_0^T\int_0^1\big(b_x( t, \bar X (t)+\lambda(X^\eps(t)-\bar X(t)), \bar u (t)+\lambda(u^\eps(t)-\bar u(t)))-b_x(t, \bar X(t), \bar u(t))\big)(X^\eps(t)-\bar X(t))d\lambda dt\bigg]
   \\&+ E\bigg[\int_0^T\int_0^1\big(b_u( t,\bar X (t)+\lambda(X^\eps(t)-\bar X(t)),\bar u (t)+\lambda(u^\eps(t)-\bar u(t)))-b_u(t, \bar X(t), \bar u(t))\big)(u^\eps(t)-\bar u(t)) dt\bigg]
   \\&+\mathbb E\bigg[\int_0^Tb_x(t, \bar X(t), \bar u(t))\big)\Xi^\eps(t)dt\bigg]
   +\eps \mathbb E\bigg[\int_0^Tb_x(t, \bar X(t), \bar u(t))\big)Y(t)d\lambda dt\bigg]
   \\&+ \eps E\bigg[\int_0^T b_u(t, \bar X(t), \bar u(t))\big)(u(t)-\bar u(t)) dt\bigg]
  \\ =& \eps \mathbb E\bigg[\int_0^Tb_x(t, \bar X(t), \bar u(t))Y(t)d\lambda dt\bigg]+ \eps E\bigg[\int_0^T b_u(t, \bar X(t), \bar u(t))(u(t)-\bar u(t)) dt\bigg]
   \\&+o(\eps),
  \end{split}
\end{eqnarray}

Similarly,  we have

\begin{eqnarray} \label{eq:6.12}
  \begin{split}
    I_1= \eps \mathbb E\bigg[\Phi _x(\bar X(T))Y(T)\bigg]+o(\eps),
  \end{split}
\end{eqnarray}

Then putting \eqref{eq:6.11} and
\eqref{eq:6.12} into \eqref{eq:6.10},
we get \eqref{eq:6.9}. The proof is complete.

\end{proof}

\subsection{Main Results}
Now we are in position to  state and prove the maximum principle for our optimal control problem.
\begin{thm}[{\bf Maximum Principle}]
\label{thm:4.3}
Let  Assumption \ref{ass:2.5} be satisfied.
Let $( {\bar u} (\cdot); {\bar X} (\cdot) )$ be an optimal pair of
Problem \ref{pro:2.1} associated with
the adjoint processes $( { p} (\cdot), {
q} (\cdot), r(\cdot, \cdot) ).$
Then the following minimum condition holds:
\begin{eqnarray}\label{eq:4.9}
\big( {\cal H}_u (t,\bar X(t-), \bar u(t),p(t-), q(t),
r(t,\cdot)), v - {\bar u} (t) \big)_U \geq 0 ,\quad \quad\forall v \in {\mathscr U}, for~ a.e.~t \in [ 0, T ], {\mathbb P}-a.s.
\end{eqnarray}
\end{thm}
\begin{proof}
Applying It\^o formula to
     $(p(t), Y(t))_H$ leads to
  \begin{eqnarray} \label{eq:6.14}
    \begin{split}
        &\mathbb E[(\Phi_x( \bar X(T)),
         Y(T))_H ]
      + \mathbb E\bigg[\int_0^T
      ( l_x(t,\bar X(t), \bar u(t) ), Y(t))_Hdt\bigg]
      \\=& \mathbb E\bigg[\int_0^T
      \bigg(v(t)-\bar u(t),  b_u^*(t,\bar X(t), \bar u(t))p(t)
      + g_u^*(t,\bar X(t), \bar u(t))q(t)
      +\int_{E}
      \sigma_u^*(t,e,\bar X(t), \bar u(t))r(t,e)\nu(de)\bigg)_Udt\bigg].
    \end{split}
  \end{eqnarray}

Since $\bar u(\cdot) $ is the optimal
control,
from \eqref{eq:6.9} and the duality relation
\eqref{eq:5.3}, we have
\begin{eqnarray}\label{eq:5.6}
  \begin{split}
  0\leq&   \lim_{\eps\longrightarrow 0}\frac {J(u^\eps(\cdot))-J(\bar u(\cdot))}{\eps}
 \\ =&\mathbb E[\big(\Phi_x( \bar X(T)),
         Y(T)\big)_H ]
      + \mathbb E\bigg[\int_0^T
      ( l_x(t,\bar X(t), \bar u(t) ), Y(t))_Hdt\bigg]+ \mathbb E\bigg[\int_0^T
      ( l_u(t, \bar X(t), \bar u(t)), v(t)-u(t))_Udt\bigg]
      \\=&\mathbb E\bigg[\int_0^T
      \bigg(v(t)-\bar u(t),  b_u^*(t,\bar X(t), \bar u(t))p(t)
      + g_u^*(t,\bar X(t), \bar u(t))q(t)
      +\int_{E}
      \sigma_u^*(t,e,\bar X(t), \bar u(t))r(t,e)\nu(de)\bigg)_Udt\bigg]
      \\&+\mathbb E\bigg[\int_0^T
      ( l_u(t, \bar X(t), \bar u(t)), v(t)-\bar u(t))_Udt\bigg]
      \\=&E\bigg[\int_0^T
      \bigg(v(t)-\bar u(t), {\cal H}_u (t,\bar X(t), \bar u(t),p(t), q(t),
r(t,\cdot)) )_Udt\bigg].
  \end{split}
\end{eqnarray}
This imply the minimum condition
\eqref{eq:4.9} holds  since $v(\cdot)$
is arbitrary given admissible control.
  \end{proof}

\section{Verification  Theorem }

In the following, we give  the sufficient condition of optimality for the existence of an optimal control of
Problem \ref{pro:2.1}.

\begin{thm}[{\bf Verification  Theorem}]\label{thm:4.4}
Let  Assumption \ref{ass:2.5} be satisfied.
Let $( {\bar u} (\cdot); {\bar X} (\cdot) )$ be an admissible pair of
Problem \ref{pro:2.1} associated with
the adjoint processes $(  p(\cdot),
q (\cdot ), r(\cdot, \cdot)).$ Suppose that 
${\cal H} ( t, x, u, { p} (t), { q} (t), r(t,\cdot) )$ is convex in $( x, u)$,
and $\Phi (x)$ is convex in $x$, moreover assume
that  the following  optimality condition holds for almost all $( t, \omega ) \in [ 0, T ] \times \Omega$:
\begin{eqnarray} \label{eq:4.10}
 {\cal H} ( t, {\bar X} (t),
\bar u(t), { p} (t), {q} (t), r(t,\cdot)) = \min_{ u \in {\mathscr U} } {\cal H} ( t, {\bar x} (t),
u, { p} (t), {q} (t), r(t,\cdot)).
\end{eqnarray}
 Then $( {\bar u} (\cdot); {\bar X} (\cdot) )$
is an optimal pair of Problem \ref{pro:2.1}.
\end{thm}
\begin{proof}
Let  $( {u} (\cdot); {X} (\cdot) )$  be an any   given admissible pairs. To simplify our notation, we define the following  shorthand notations:
\begin{eqnarray} \label{eq:6.1}
\begin{split}
& b (t) \triangleq b ( t, X (t), u (t)), {\bar b} (t) \triangleq b ( t, {\bar X} (t), {\bar u} (t)), \\
& g (t) \triangleq g ( t, X (t), u (t)), {\bar g} (t) \triangleq g ( t, {\bar X} (t) ,{\bar u} (t)),
\\
& \sigma (t,e) \triangleq \sigma ( t,e, X (t), u (t)), {\bar \sigma} (t) \triangleq \sigma ( t,
e,{\bar X} (t) ,{\bar u} (t)),
\\
& {\cal H} (t) \triangleq {\cal H} ( t, X (t), u (t), {p} (t), {q} (t), r(t,\cdot) ),
\\
& { \bar{\cal H}} (t) \triangleq {\cal H} ( t,
\bar X (t), \bar u (t), {p} (t), {q} (t), r(t,\cdot) ) .
\end{split}
\end{eqnarray}
From the definitions of the cost functional
  $J(u(\cdot))$ and the Hamiltonian  ${\cal H}$ (see \eqref{eq:4.2} and \eqref{eq:5.3}),
  we can represent $J(u(\cdot))-J(\bar u(\cdot))$ as follows:
\begin{eqnarray}\label{eq:4.10}
J ( u (\cdot) ) - J ( {\bar u} (\cdot) )
&=& {\mathbb E} \bigg [ \int_0^T \bigg \{ {\cal H} (t) - {\bar {\cal H}} (t)
- ( {\bar p} (t), b(t) - {\bar b} (t) )_H \nonumber - ( {\bar q} (t), g (t) - {\bar g} (t) )_H
\\&&-\int_{E}( {\bar r} (t,e), \sigma (t,e) - {\bar \sigma} (t,e) )_H  \nu(de) \bigg \} d t \bigg ]+ {\mathbb E} \bigg[ \Phi ( X (T))
- \Phi ( {\bar X} (T) )\bigg] .
\end{eqnarray}
\end{proof}

 In terms of  the state equation \eqref{eq:3.1}, we can check that
$X(\cdot)-\bar X(\cdot)$ satisfies the
following SEE:
\begin{eqnarray}\label{eq:5.10}
\left\{
\begin{aligned}
d (X (t)-\bar X(t)) =& \ [  A (t) (X(t)
 -\bar X(t))+ b ( s )-\bar b ( s ) ] d t
 +[ B(t)(X(t)-\bar X(t))+ g ( s )-\bar g(s) )] d W (t)
 \\&+\int_{E}[\sigma ( s,e )-\bar \sigma(s,e) )] d \tilde \mu (de,t) , \quad t \in [ 0, T ] , \\
X(0)-\bar X(0) =& \ 0.
\end{aligned}
\right.
\end{eqnarray}
Then recalling the
 adjoint equation \eqref{eq:4.4}  and applying  It\^o's formula to $( {p} (t), X (t) - {\bar X} (t) )_H$, we  get that
\begin{eqnarray}\label{eq:4.12}
&&{\mathbb E}\bigg [ \int_0^T  \bigg \{ ( { p} (t), b (t) - {\bar b} (t) )_H
+ ( {q} (t), g (t) - {\bar g} (t) )_H
+\int_{E}( { r} (t,e), \sigma (t,e) - {\bar \sigma} (t,e) )_H  \nu(de)
\bigg \} d t \bigg ] \nonumber \\
~~~~~~~~~&=& {\mathbb E} \bigg [ \int_0^T ( {\bar {\cal H}}_x (t) , X (t) - {\bar X} (t) )_H d t \bigg ]+ {\mathbb E} \Big[ (\Phi_{x} ( \bar X  (T)), X (T) - {\bar X} (T) )_H \Big] .
\end{eqnarray}
Then substituting \eqref{eq:4.12} into \eqref{eq:4.10} leads to

\begin{eqnarray}\label{eq:6.6}
J ( u (\cdot) ) - J ( {\bar u} (\cdot) )
&=& {\mathbb E} \bigg [ \int_0^T \bigg \{ {\cal H} (t) - {\bar {\cal H}} (t)
- ( {\bar {\cal H}}_x (t),
 X(t) - {\bar X} (t) )_H \bigg \} d t \bigg ] \nonumber \\
&& + {\mathbb E} [ \Phi ( X (T) ) - \Phi ( {\bar X} (T) )
- ( \Phi_x ( {\bar X} (T) ), X(T) - {\bar x} (T) )_H ] .
\end{eqnarray}
On the other hand,   the convexity of
${\cal H}(t)$ and $\Phi (x)$ leads to
\begin{eqnarray}\label{eq:7.7}
{\cal H} (t) - {\bar {\cal H}} (t) &\geq& ( {\bar {\cal H}}_x (t), X (t) - {\bar X} (t) )_H
+( {\bar {\cal H}}_{u} (t), u ( t ) - {\bar u} ( t) )_U ,
\end{eqnarray}
and
\begin{eqnarray}\label{eq:5.4}
\Phi ( X (T) ) - \Phi ( {\bar X} (T) ) \geq ( \Phi_x ( {\bar X} (T) ), x (T) - {\bar x} (T) )_H .
\end{eqnarray}
In addition,  the optimality condition \eqref{eq:4.10}  and the convex optimization principle
(see Proposition 2.21 of [8] ) yield  that  for almost all $( t, \omega ) \in
[ 0, T ] \times \Omega$,
\begin{eqnarray}\label{eq:5.5}
( {\bar {\cal H}}_u (t), u (t) - {\bar u} (t) )_U\geq 0 .
\end{eqnarray}
Then putting \eqref{eq:7.7}, \eqref{eq:5.4} and \eqref{eq:5.5} into \eqref{eq:6.6}, we get that
\begin{eqnarray}
 J ( u (\cdot) ) - J ( {\bar u} (\cdot) )\geq 0.
\end{eqnarray}
Therefore, since $u(\cdot)$ is arbitrary,  ${\bar u} (\cdot)$ is an optimal control process and $( {\bar u} (\cdot); {\bar x} (\cdot) )$ is an optimal pair.  The proof is
complete.

\section{Application}
We provide an example to which our results solve.  We consider a controlled  Cauchy problem, where the system
is given by a stochastic partial differential equation driven by
Brownian motion $W$
and Poisson random
martingale $\tilde \mu$ in
divergence form:
\begin{eqnarray}\label{eq:6.13}
\left\{
\begin{aligned}
d y ( t,z ) = & \ \big \{ \partial_{z^i} [ a^{ij} ( t, z )
\partial_{z^j} y ( t, z ) ]
+ b^i ( t, z ) \partial_{z^i} y ( t, z )+ c ( t, z ) y ( t, z ) + u ( t, z ) \big \} d t \\
& + \{\partial_{z^i} [\eta^i( t, z )y ( t, z )]  +\rho ( t, z ) y ( t, z )+
u ( t, z ) \} d W (t)
+ \int_{E}[  \Gamma ( t, e,z ) y ( t, z )+
u ( t, z ) ] \tilde \mu(de,dt) ,\\
y ( 0, z ) = & \ \xi(z)\in \mathbb R^d \quad ( t, z ) \in [ 0, T ] \times {\mathbb R}^d,
\end{aligned}
\right.
\end{eqnarray}
with a quadratic cost functional
\begin{eqnarray}\label{eq:8.2}
{\mathbb E} \bigg [ \int_{{\mathbb R}^d} y^2 ( T, z ) d z + \iint_{[
0, T ] \times {{\mathbb R}^d}} y^2 ( s, z ) d z  d s + \iint_{[ 0, T
] \times {{\mathbb R}^d}} u^2 ( s, z ) d z d s \bigg ] .
\end{eqnarray}
Here the unknown  $y(t,z,\omega)$, representing the state of the system, is a real-valued
process, the control is a predictable real-valued process $u(t,z,\omega).$
The coefficients $a$, $b$, $c$, $\eta$, $\rho$, $\Gamma$ are given  random functions satisfying
the following assumptions, for some fixed constants $K \in ( 1,
\infty )$ and $\kappa \in ( 0,1 )$:

\begin{ass}\label{ass:6.3}
The functions $a$, $b$, $c$, $\eta$, and $\rho$  are
${\mathscr P} \times {\mathscr B} ({\mathbb R}^d )$-measurable with
values in the set of real symmetric $d \times d$ matrices ${\mathbb
R}^{d}$, ${\mathbb R}$, ${\mathbb R}^d$
and  ${\mathbb R}$, respectively, and are bounded by $K$.
The function $\Gamma$ are
${\mathscr P}
\times {\mathscr B} (E)\times {\mathscr B} ({\mathbb R}^d )$-measurable with
value ${\mathbb
R}$ and is bounded by $K$. $\xi\in L^2 ( {\mathbb R}^d).$
\end{ass}

\begin{ass}\label{ass:6.4}
The super-parabolic condition holds, i.e.,
\begin{eqnarray*}
\kappa I+\eta( t, z )(\eta( t, z ))^*\leq 2 a ( t, \omega, z ) \leq K I , \quad \forall ( t, \omega, z ) \in
[ 0, T ] \times \Omega \times {\mathbb R}^d ,
\end{eqnarray*}
where $I$ is the $(d \times d)$-identity matrix.
\end{ass}

Now we  begin to transform
\eqref{eq:6.13} into a SEE  with  jump  in
the form of \eqref{eq:3.1}. To this end,
let us recall some preliminaries of Sobolev spaces.
For $m = 0, 1$, we define the space $H^m \triangleq \{ \phi:
\partial_z^\alpha \phi \in L^2 ( {\mathbb R}^d ), \ \mbox {for any}
\ \alpha: =( \alpha_1, \cdots, \alpha_d ) \ \mbox {with} \ |\alpha|
:= | \alpha_1 | + \cdots + | \alpha_d | \leq m \}$ with the norm
\begin{eqnarray*}
\| \phi \|_m \triangleq \left \{ \sum_{ |\alpha| \leq m } \int_{{\mathbb R}^d}
| \partial_z^\alpha \phi (z) |^2 d z \right \}^{\frac{1}{2}} .
\end{eqnarray*}
We denote by $H^{-1}$ the dual space of $H^1$. We set $V = H^1$, $H
= H^0$, $V^* = H^{-1}$. Then $( V, H, V^* )$ is a Gelfand triple.

In our case,  we assume control domain ${\mathscr U} = U = H$. The admissible control
set $\cal A$  is defined as  $ { M}^2_{\mathscr F} ( 0, T; U ).$
Set
\begin{eqnarray*}
&&X(t)\triangleq y(t, \cdot),
\\&& (A (t) \phi) (z) \triangleq  \partial_{z^i} [ a^{ij} ( t, z )
\partial_{z^j} \phi (z) ]
+b^i ( t, z ) \partial_{z^i} \phi (z)+c ( t, z ) \phi (z)  , \quad \forall \phi \in V , \\
&&( B (t) \phi) (z) \triangleq \partial_{z^i}[\eta^i (t,z)\phi (z)]+\rho ( t, z ) \phi ( z ) ,\quad \forall \phi \in V,
\\&& b(t,\phi, u)\triangleq u, \quad \forall \phi \in H, u\in {\cal U},
\\&& g(t,\phi, u)\triangleq u, \quad \forall \phi \in H, u\in {\cal U},
\\&& \sigma(t,e,\phi, u)\triangleq\Gamma ( t, e,\cdot ) \phi+u,  \quad\forall \phi \in H, u\in \cal U,
 \\&& l(t,\phi, u)\triangleq  (\phi, \phi)_H+
 (u, u)_{U},  \quad\forall \phi \in H, u\in \cal U,
 \\&& \Phi(\phi)\triangleq (\phi, \phi)_H, \quad\forall \phi \in H.
\end{eqnarray*}

In the Gelfand triple $( V, H, V^* )$, using the above notations, we can rewrite the state equation \eqref{eq:6.13} as follows:
\begin{eqnarray} \label{eq:8.3}
  \left\{
  \begin{aligned}
   d X (t) = & \ [ A (t) X (t) + b ( t, X (t), u(t)) ] d t
+ [B(t)X(t)+g( t, X (t), u(t)) ]d W(t)
 \\&\quad +\int_E \sigma (t, e,X(t),u(t))\tilde \mu(de,dt),  \\
X (0) = & \  x , \quad t \in [ 0, T ],
  \end{aligned}
  \right.
\end{eqnarray}
and the cost functional \eqref{eq:8.2}
can be  rewritten as
\begin{eqnarray}\label{eq:8.4}
J ( u (\cdot) ) = {\mathbb E} \bigg [ \int_0^T l ( t, x (t), u (t) ) d t
+ \Phi ( x (T) ) \bigg ].
\end{eqnarray}
where we set

\begin{eqnarray}
\begin{split}
  &l(t,x,u)\triangleq(x, x)_H+(u, u)_H, \forall x\in H, u\in U,
\\&
\Phi(x)\triangleq(x,x)_H, \forall x\in H.
\end{split}
\end{eqnarray}

Thus this optimal control problem is  transformed  into Problem
\ref{pro:2.1} as a special case.
Under Assumptions \ref{ass:6.3}-\ref{ass:6.4},
 it  is easy to check that the coefficients of  this optimal control problem  satisfy
Assumptions \ref{ass:2.5}.  So  in this
case,  Theorem \ref{thm:4.3} and
\ref{thm:4.4} hold.   More precisely, the corresponding Hamiltonian ${\cal H}$ becomes
\begin{eqnarray}\label{eq4.2}
{\cal H} ( t, x, u, p, q, r ) := \left ( u, p \right )_H
+\left(u , q \right)_H
+\int_{E}\left( \Gamma ( t, e,\cdot ) x+u, r(t,e) \right)_H\nu(de)
+ (x,x)_H+(u,u)_H .
\end{eqnarray}
The adjoint equation  associated with  the optimal pair $( \bar u (\cdot); \bar X (\cdot) )$ becomes

\begin{eqnarray}\label{eq:8.6}
\begin{split}
   \left\{\begin{array}{ll}
dp(t)=&\bigg[A^*(t)p(t)
+B^*(t)q(t)+\displaystyle\int_{{E}}
\Gamma^*(t,e)r(t, e)\nu (de)dt
+ 2 X(t)\bigg]dt
\\&\quad\quad+q(t)dW(t)+\displaystyle
\int_{{E}}r(t, e)\tilde{\mu}(de, dt),
~~~~0\leqslant t\leqslant T,
\\p(T)=&\Phi_x(T),
  \end{array}
 \right.
 \end{split}
  \end{eqnarray}
where
\begin{eqnarray*}
 &&A^*(t) \phi (z) \triangleq - \partial_{z^i} [ a^{ij}
( t, z ) \partial_{z^j} \phi (z) ] + \partial_{z^i}[b^i ( t, z )
\phi (z)] + c ( t, z ) \phi (z) , \quad \forall \phi \in V,\\
&&B^*(t) \phi (z) \triangleq   -\eta^i (t,z)\partial_{z^i}\phi (z),  \quad \forall  \phi \in H,\\
&&\Gamma^*(t,e) \phi (z) \triangleq   \Gamma(t,e,z)\phi (z),  \quad \forall  \phi \in H.
\end{eqnarray*}
Since ${\mathscr U} = U$, there is no constraint on the control and  the minimum condition \eqref{eq:4.9}
\begin{eqnarray}
{\cal H}_u (t,\bar X(t-), \bar u(t),p(t-), q(t),
r(t,\cdot))= 0 .
\end{eqnarray}
which imply that
\begin{eqnarray}\label{eq6.4}
&& 2 {\bar u} (t) + { p} (t) + q (t)
+\int_{E} r(t,e)\nu(de)= 0,
\end{eqnarray}a.e. $t \in [ 0, T ]$, ${\mathbb P}$-a.s..
 Thus  the optimal control ${\bar u}
(\cdot)$ is given by
\begin{eqnarray*}
{\bar u} (t) = -\frac{1}{2} \bigg [ {\bar p} (t) + {\bar q} (t)+\int_{E} r(t,e)\nu(de)\bigg] .
\end{eqnarray*}

\end{document}